\documentclass[12pt,draftcls,onecolumn]{IEEEtran}
\usepackage{amsfonts}
\usepackage{amsmath}
\usepackage{amssymb}
\usepackage{color}
\usepackage[usenames,dvipsnames]{xcolor}
\usepackage{epsfig}
\usepackage{graphics}
\usepackage{graphicx}

\epsfclipon
\newtheorem{theorem}{Theorem}[section]

\newtheorem{definition}{Definition}[section]
\newtheorem{example}{Example}[section]

\newtheorem{lemma}[theorem]{Lemma}

\newtheorem{problem}{Problem}
\newtheorem{proposition}{Proposition}
\newtheorem{remark}{Remark}[section]

\begin{document}

\title{Structural Analysis of Laplacian Spectral Properties 
of Large-Scale Networks}
\author{Victor M. Preciado,~\IEEEmembership{Member,~IEEE,} Ali Jadbabaie,~%
\IEEEmembership{Senior~Member,~IEEE}, and George C. Verghese, %
\IEEEmembership{Fellow,~IEEE} \thanks{%
V.M. Preciado and A. Jadbabaie are with the Department of Electrical and
Systems Engineering at the University of Pennsylvania, Philadelphia, PA
19104 USA. (e-mail: preciado@seas.upenn.edu; jadbabai@seas.upenn.edu). } 
\thanks{%
G.C. Verghese is with the Department of Electrical Engineering and Computer
Science at the Massachusetts Institute of Technology, Cambridge, MA $02139$
USA. (e-mail: verghese@mit.edu). } \thanks{%
This work was supported by ONR MURI \textquotedblleft Next Generation
Network Science\textquotedblright\ and AFOSR \textquotedblleft Topological
and Geometric Tools for Analysis of Complex Networks\textquotedblright .}}
\maketitle

\begin{abstract}
Using methods from algebraic graph theory and convex optimization,
we study the relationship between local structural features of a network and
spectral properties of its Laplacian matrix. In particular, we derive
expressions for the so-called spectral moments of the Laplacian matrix of a
network in terms of a collection of local structural measurements. Furthermore, we propose a
series of semidefinite programs to compute bounds on the spectral radius and
the spectral gap of the Laplacian matrix from a truncated sequence of
Laplacian spectral moments. Our analysis shows that the Laplacian spectral
moments and spectral radius are strongly constrained by local structural
features of the network. On the other hand, we illustrate how local structural features are usually
not enough to estimate the Laplacian spectral gap.
\end{abstract}

\section{Introduction}

Understanding the relationship between the structure of a network
and the behavior of dynamical processes taking place in it is a central
question in the research field of network science \cite{NBW06}. Since the
behavior of many networked dynamical processes is closely related with the
Laplacian eigenvalues (see \cite{FM04}, \cite{OFM07} and references
therein), it is of interest to study the relationship between structural
features of the network and its Laplacian eigenvalues.

In this technical note, we study this relationship, focusing on the
role played by structural features that can be extracted from localized
samples of the network structure. Our objective is then to efficiently
aggregate these local samples of the network structure to infer global
properties of the Laplacian spectrum. We propose a graph-theoretical
approach to relate structural features of a network with algebraic
properties of its Laplacian matrix. Our analysis reveals that there are
certain spectral properties, such as the so-called spectral moments, that
can be efficiently computed from these structural features. Furthermore,
applying a recent result by Lasserre \cite{Las11}, we propose a series of
semidefinite programs to compute bounds on the Laplacian spectral radius and
spectral gap from a truncated sequence of spectral moments.

The paper is organized as follows. In the next subsection, we define
terminology needed in our derivations. In Section \ref{Laplacian Spectral
Analysis}, we introduce a graph-theoretical methodology to derive
closed-form expressions for the so-called Laplacian spectral moments in
terms of structural features of the network. In Section \ref{Optimal Bounds}, we use semidefinite programming to derive optimal bounds on the Laplacian
spectral radius and spectral gap from a truncated sequence of spectral
moments. We validate our results numerically in Section \ref{Simulations}.

\subsection{\label{Notation}Notations \& Preliminaries}

Let $\mathcal{G}=\left( \mathcal{V},\mathcal{E}\right) $ be an undirected
graph, where $\mathcal{V}=\left\{ v_{1},\dots ,v_{n}\right\} $ denotes a set
of $n$ nodes and $\mathcal{E}\subseteq \mathcal{V}\times \mathcal{V}$
denotes a set of $e$ undirected edges. If $\left\{ v_{i},v_{j}\right\} \in 
\mathcal{E}$, we call nodes $v_{i}$ and $v_{j}$ \emph{adjacent} (or
first-neighbors), which we denote by $v_{i}\sim v_{j}$. We define the set of
first-neighbors of a node $v_{i}$ as $\mathcal{N}_{i}=\{w\in \mathcal{V}%
:\left\{ v_{i},w\right\} \in \mathcal{E}\}.$ The \emph{degree} $d_{i}$ of a
vertex $v_{i}$ is the number of nodes adjacent to it, i.e., $%
d_{i}=\left\vert \mathcal{N}_{i}\right\vert $. We consider three types of
undirected graphs: (\emph{i}) A graph is called \emph{simple} if its edges
are unweighted and it has no self-loops\footnote{%
A self-loop is an edge of the type $\left\{ v_{i},v_{i}\right\} $.}, (\emph{%
ii}) a graph is \emph{loopy} if it has self-loops, and (\emph{iii}) a graph
is \emph{weighted} if there is a real number associated with every edge in
the graph. More formally, a weighted graph $\mathcal{H}$ can be defined as
the triad $\mathcal{H=}\left( \mathcal{V},\mathcal{E},\mathcal{W}\right) $,
where $\mathcal{V}$ and $\mathcal{E}$ are the sets of nodes and edges in $%
\mathcal{H}$, and $\mathcal{W=}\left\{ w_{ij}\in \mathbb{R},\text{ for all }%
\left\{ v_{i},v_{j}\right\} \in \mathcal{E}\right\} $ is the set of
(possibly negative) weights.

The \emph{adjacency matrix} of a simple graph $\mathcal{G}$, denoted by $A_{%
\mathcal{G}}=[a_{ij}]$, is an $n\times n$ symmetric matrix defined
entry-wise as $a_{ij}=1$ if nodes $v_{i}$ and $v_{j}$ are adjacent, and $%
a_{ij}=0$ otherwise. In the case of weighted graphs (and possibly
non-simple), the weighted adjacency matrix is defined by $W_{\mathcal{G}}=%
\left[ w_{ij}\right] $, where $w_{ij}=0$ if $v_{i}$ is not adjacent to $%
v_{j} $. We define the \emph{degree matrix} of a simple graph $\mathcal{G}$
as the diagonal matrix $D_{\mathcal{G}}=diag\left( d_{i}\right) $. We define
the \emph{Laplacian matrix }$L_{\mathcal{G}}$ (also known as combinatorial
Laplacian, or Kirchhoff matrix) of a simple graph as $L_{\mathcal{G}}=D_{%
\mathcal{G}}-A_{\mathcal{G}}$. For simple graphs, $L_{\mathcal{G}}$ is a
symmetric, positive semidefinite matrix, which we denote by $L_{\mathcal{G}%
}\succeq 0$ \cite{Big93}. Thus, $L_{\mathcal{G}}$ has a full set of $n$ real
and orthogonal eigenvectors with real nonnegative eigenvalues $0=\lambda
_{1}\leq \lambda _{2}\leq ...\leq \lambda _{n}$. The second smallest and
largest eigenvalues of $L_{\mathcal{G}}$, $\lambda _{2}$ and $\lambda _{n}$,
are called the spectral gap and spectral radius of $L_{\mathcal{G}}$,
respectively. Given a $n\times n$ real and symmetric matrix $B$ with (real)
eigenvalues $\sigma _{1},...,\sigma _{n}$, we define the $k$-th spectral
moment of $B$ as $m_{k}\left( B\right) \triangleq \frac{1}{n}%
\sum_{i=1}^{n}\sigma _{i}^{k}.$ As we shall show in Section \ref{Laplacian
Spectral Analysis}, there is an interesting connection between the spectral
moments of the Laplacian matrix, $m_{k}\left( L_{\mathcal{G}}\right) $, and
structural features of the network.

We now define a collection of structural properties that are important in
our derivations. The \emph{degree sequence} of a simple graph $\mathcal{G}$
is the ordered list of its degrees, $\left( d_{1},...,d_{n}\right) $. A 
\emph{walk} of length $k$ from $v_{i_{1}}$ to $v_{i_{k+1}}$ is an ordered
sequence of nodes $\left( v_{i_{1}},v_{i_{2}},...,v_{i_{k+1}}\right) $ such
that $v_{i_{j}}\sim v_{i_{j+1}}$ for $j=1,2,...,k$. One says that the walk 
\emph{touches} each of the nodes that comprises it. If $%
v_{i_{1}}=v_{i_{k+1}} $, then the walk is closed. A closed walk with no
repeated nodes (with the exception of the first and last nodes) is called a 
\emph{cycle}. Given a walk $p=\left(
v_{i_{1}},v_{i_{2}},...,v_{i_{k+1}}\right) $ in a weighted graph $\mathcal{H}
$, we define the weight of the walk as, $\omega \left( p\right)
=w_{i_{1}i_{2}}w_{i_{2}i_{3}}...w_{i_{k}i_{k+1}}$.

\section{\label{Laplacian Spectral Analysis}Moment-Based Analysis of the
Laplacian Matrix}

In this paper, we use algebraic graph theory to study the
relationship between structural properties of a network and its Laplacian
spectrum based on the spectral moments. A well-known result in algebraic
graph theory relates the diagonal entries of the $k$-th power of
the adjacency matrix, $\left[ A_{\mathcal{G}}^{k}\right] _{ii}$, to
the number of closed walks of length $k$ in $G$ that
start and finish at node $v_{i}$ \cite{Big93}. Using this result,
it is possible to relate algebraic properties of the adjacency matrix $A_{%
\mathcal{G}}$ to the presence of certain subgraphs in the network 
\cite{PJ10}. We can generalize this result to weighted graphs as follows:

\medskip

\begin{proposition}
\label{Moments from Walks}Let $\mathcal{H=}\left( \mathcal{V},\mathcal{E},%
\mathcal{W}\right) $ be a weighted graph with weighted adjacency matrix $W_{%
\mathcal{H}}=\left[ w_{ij}\right] $. Then, $\left[ W_{\mathcal{H}}^{k}\right]
_{ii}=\sum_{p\in P_{k,i}}\omega \left( p\right) ,$where $P_{k,i}$ is the set
of closed walks of length $k$ from $v_{i}$ to itself in $\mathcal{H}$.
\end{proposition}

\medskip

\begin{proof}
By recursively applying the multiplication rule for matrices, we have the
following expansion%
\begin{equation}
\left[ W_{\mathcal{H}}^{k}\right] _{ii}=\sum_{i=1}^{n}\sum_{i_{2}=1}^{n}%
\cdots \sum_{i_{k}=1}^{n}w_{i,i_{2}}w_{i_{2,}i_{3}}\cdots ~w_{i_{k},i}.
\label{Multiple Summations}
\end{equation}%
Using the graph-theoretic nomenclature introduced in Section \ref{Notation},
we have that $w_{i,i_{2}}w_{i_{2,}i_{3}}...w_{i_{k},i}=\omega \left(
p\right) $, for $p=\left(
v_{i},v_{i_{2}},v_{i_{3}},...,v_{i_{k}},v_{i}\right) $. Hence, the
summations in (\ref{Multiple Summations}) can be written as $\left[ W_{%
\mathcal{H}}^{k}\right] _{ii}=\sum_{1\leq i,i_{2},...,i_{k}\leq n}\omega
\left( p\right) $. Finally, the set of closed walks $p=\left(
v_{i},v_{i_{2}},v_{i_{3}},...,v_{i_{k}},v_{i}\right) $\ with indices $1\leq
i,i_{2},...,i_{k}\leq n$ is equal to the set of closed walks of length $k$
from $v_{i}$ to itself in $\mathcal{H}$ (which we have denoted by $P_{k,i}$
in the statement of the Proposition).
\end{proof}

\medskip

The above Proposition allows us to write the relate moments of the weighted
adjacency matrix of a weighted graph $\mathcal{H}$ to closed walks in $%
\mathcal{H}$, as follows:

\medskip

\begin{lemma}
\label{Moments in Weighted Graphs}Let $\mathcal{H=}\left( \mathcal{V},%
\mathcal{E},\mathcal{W}\right) $ be a weighted graph with weighted adjacency
matrix $W_{\mathcal{H}}=\left[ w_{ij}\right] $. Then, 
\begin{equation*}
m_{k}\left( W_{\mathcal{H}}\right) =\frac{1}{n}\sum_{v_{i}\in \mathcal{V}%
}\sum_{p\in P_{k,i}}\omega \left( p\right) ,
\end{equation*}%
where $P_{k,i}$ is the set of closed walks of length $k$ from $v_{i}$ to
itself in $\mathcal{H}$.
\end{lemma}

\medskip

\begin{proof}
Let us denote by $\mu _{1},...,\mu _{n}$ the set of (real) eigenvalues of
the (symmetric) weighted adjacency matrix $W_{\mathcal{H}}$. We have that
the moments can be written as 
\begin{equation*}
m_{k}\left( W_{\mathcal{H}}\right) \triangleq \frac{1}{n}\sum_{i=1}^{n}\mu
_{i}^{k}=\frac{1}{n}\text{Trace}\left( W_{\mathcal{H}}^{k}\right) ,
\end{equation*}%
since $W_{\mathcal{H}}$ is a symmetric (and diagonalizable) matrix. We then
apply Proposition \ref{Moments from Walks} to rewrite the moments as follows,%
\begin{equation*}
m_{k}\left( W_{\mathcal{H}}\right) =\frac{1}{n}\sum_{i=1}^{n}\left[ W_{%
\mathcal{H}}^{k}\right] _{ii}=\frac{1}{n}\sum_{v_{i}\in \mathcal{V}%
}\sum_{p\in P_{k,i}}\omega \left( p\right) .
\end{equation*}
\end{proof}

\medskip

In Subsections \ref{High Order Moment}, we shall apply this result to
compute spectral moments of the Laplacian matrix in terms of structural
features of the network. First, we need to introduce a weighted graph that
is useful in our derivations:

\medskip

\begin{definition}
\label{Laplacian Graph}Given a simple graph $\mathcal{G=}\left( \mathcal{V},%
\mathcal{E}\right) $, we define the Laplacian graph of $\mathcal{G}$ as the
weighted graph $\mathcal{L}\left( \mathcal{G}\right) \mathcal{\triangleq }%
\left( \mathcal{V},\mathcal{E\cup S}_{n},\Gamma \right) $, where $\mathcal{S}%
_{n}=\left\{ \left\{ v,v\right\} \text{ for all }v\in \mathcal{V}\right\} $
(the set of all self-loops), and $\Gamma =\left[ \gamma _{ij}\right] $ is a
set of weights defined as:%
\begin{equation*}
\gamma _{ij}\triangleq \left\{ 
\begin{array}{ll}
-1, & \text{for }\left\{ v_{i},v_{j}\right\} \in \mathcal{E} \\ 
d_{i}, & \text{for }i=j \\ 
0, & \text{otherwise.}%
\end{array}%
\right.
\end{equation*}
\end{definition}

\begin{remark}
Note that the weighted adjacency matrix of the Laplacian graph $\mathcal{L}%
\left( \mathcal{G}\right) $ is equal to the Laplacian matrix of the simple
graph $\mathcal{G}$. Hence, we can apply Lemma \ref{Moments in Weighted
Graphs} to express the spectral moments of the Laplacian matrix $L_{\mathcal{%
G}}$ in terms of weighted walks in the Laplacian graph $\mathcal{L}\left( 
\mathcal{G}\right) $.
\end{remark}

Before we apply Lemma \ref{Moments in Weighted Graphs} to study the
Laplacian spectral moments, we must introduce the concept of subgraph
covered by a walk.

\begin{definition}
\label{Covered Subgraph}Consider a walk $p=\left(
v_{i_{1}},v_{i_{2}},...,v_{i_{k+1}}\right) $ of length $k$ in a (possibly
loopy) graph. We define the subgraph covered by $p$ as the simple graph $%
C\left( p\right) =(\mathcal{V}_{c}\left( p\right) ,\mathcal{E}_{c}\left(
p\right) )$, with node-set $\mathcal{V}_{c}\left( p\right)
=\bigcup_{r=1}^{k+1}v_{i_{r}}$, and edge-set $\mathcal{E}_{c}\left( p\right)
=\bigcup_{v_{i_{r}}\neq v_{i_{r+1}}}\left\{ v_{i_{r}},v_{i_{r+1}}\right\} $,
for $1\leq r\leq k+1.$
\end{definition}

Based on the above, we define \emph{triangles}, \emph{quadrangles} and \emph{%
pentagons} as the subgraphs covered by cycles of length three, four, and
five, respectively. Notice that self-loops are excluded from $\mathcal{E}%
_{c}\left( p\right) $ in Definition \ref{Covered Subgraph}. For example,
consider a walk $p=\left(
v_{1},v_{2},v_{2},v_{3},v_{3},v_{1},v_{3},v_{1}\right) $ in a graph with
self-loops. Then, $C\left( p\right) $ has node-set $\mathcal{V}_{c}\left(
p\right) =\left\{ v_{1},v_{2},v_{3}\right\} $ and edge-set $\mathcal{E}%
_{c}\left( p\right) =\left\{ \left\{ v_{1},v_{2}\right\} ,\left\{
v_{2},v_{3}\right\} ,\left\{ v_{3},v_{1}\right\} \right\} $. In other words, 
$C\left( p\right) $ is a simple triangle.

In what follows, we build on the above results to derive closed-form
expressions for the first five spectral moments of the Laplacian matrix in
terms of relevant structural features of the network.

\subsection{\label{Low Order Moments}Low-Order Laplacian Spectral Moments}

The following theorem, proved in \cite{PV08} via algebraic techniques,
allows us to compute the first three Laplacian spectral moments in terms of
the degree sequence and the number of triangles in the graph.

\begin{theorem}
\label{Low Order Laplacian Moments}Let $\mathcal{G}$ be a simple graph with
Laplacian matrix $L_{\mathcal{G}}$. Then, the first three spectral moments
of the Laplacian matrix are%
\begin{align}
m_{1}\left( L_{\mathcal{G}}\right) & =\frac{1}{n}S_{1},
\label{Moments as Averages} \\
m_{2}\left( L_{\mathcal{G}}\right) & =\frac{1}{n}\left( S_{1}+S_{2}\right) ,
\notag \\
m_{3}\left( L_{\mathcal{G}}\right) & =\frac{1}{n}\left( 3S_{2}+S_{3}-6\Delta
\right) ,  \notag
\end{align}%
where $S_{p}\triangleq \sum_{v_{i}\in \mathcal{V}}d_{i}^{p}$, and $\Delta $
is the total number of triangles in $\mathcal{G}$.
\end{theorem}

\medskip

In \cite{PV08}, Theorem \ref{Low Order Laplacian Moments} was proved
using a purely algebraic approach. This algebraic approach presents the
limitation of not being applicable to compute moments of order greater than
three. In what follows, we propose an alternative graph-theoretical approach
that allows to compute higher-order spectral moments of the Laplacian
matrix, beyond the third order. In particular, according to Lemma \ref{Moments in Weighted Graphs}, we can compute the %
$k$-th spectral moment of the Laplacian of $\mathcal{G}$ by analyzing the set of
closed walks of length $k$ in the Laplacian graph $\mathcal{L}\left( \mathcal{G}\right) $.

\subsection{\label{High Order Moment}Higher-Order Laplacian Spectral Moments}

In this Subsection, we apply the set of graph-theoretical tools introduced
above to compute the fourth- and fifth-order spectral moments of the
Laplacian matrix. We first define the collection of structural measurements
that are involved in our expressions. Let us denote by $t_{i}$, $q_{i}$, and 
$p_{i}$ the number of triangles, quadrangles, and pentagons touching node $%
v_{i}$ in $\mathcal{G}$, respectively. The total number of quadrangles and
pentagons in $\mathcal{G}$ are denoted by $Q$ and $P$, respectively. The
following terms define structural correlations that are relevant in our
analysis:%
\begin{eqnarray}
C_{dd} &\triangleq &\frac{1}{n}\sum_{v_{i}\sim v_{j}}d_{i}d_{j},\quad
C_{d^{2}d}\triangleq \frac{1}{n}\sum_{v_{i}\sim v_{j}}d_{i}^{2}d_{j},
\label{Correlation terms} \\
C_{dt} &\triangleq &\frac{1}{n}\sum_{v_{i}\in \mathcal{V}}d_{i}t_{i},\quad
C_{d^{2}t}\triangleq \frac{1}{n}\sum_{v_{i}\in \mathcal{V}}d_{i}^{2}t_{i}, 
\notag \\
C_{dq} &\triangleq &\frac{1}{n}\sum_{v_{i}\in \mathcal{V}}d_{i}q_{i},\quad
D_{dd}\triangleq \frac{1}{n}\sum_{v_{i}\sim v_{j}}d_{i}d_{j}\left\vert 
\mathcal{N}_{i}\cap \mathcal{N}_{j}\right\vert ,  \notag
\end{eqnarray}%
where $\left\vert \mathcal{N}_{i}\cap \mathcal{N}_{j}\right\vert $ is the
number of common neighbors shared by $v_{i}$ and $v_{j}$. The main result in
this section relates the fourth and fifth Laplacian spectral moments to
local structural measurements and correlation terms, as follows:

\begin{theorem}
\label{High Order Laplacian Moments}Let $\mathcal{G}$ be a simple graph with
Laplacian matrix $L_{\mathcal{G}}$. Then, the fourth and fifth Laplacian
moments can be written as%
\begin{eqnarray}
m_{4}\left( L_{\mathcal{G}}\right) &=&\frac{1}{n}\left(
-S_{1}+2S_{2}+4S_{3}+S_{4}+8Q\right)  \label{Fourth and Fifth Moments} \\
&&+4C_{dd}-8C_{dt},  \notag \\
m_{5}\left( L_{\mathcal{G}}\right) &=&\frac{1}{n}\left(
-5S_{2}+5S_{3}+5S_{4}+S_{5}+30\Delta -10P\right)  \notag \\
&&+10\left( C_{dd}+C_{d^{2}d}-C_{dt}-C_{d^{2}t}+C_{dq}-D_{dd}\right)  \notag
\end{eqnarray}%
where $S_{p}\triangleq \sum_{v_{i}\in \mathcal{V}}d_{i}^{p}$, and the
correlation terms $C_{dd},$ $C_{dt},$ $C_{dq},$ $C_{d^{2}d},$ $C_{d^{2}t},$
and $D_{dd}$ are defined in (\ref{Correlation terms}).
\end{theorem}

\begin{proof}
In the Appendix.
\end{proof}

\medskip

\begin{remark}
Theorem \ref{Low Order Laplacian Moments} relates purely algebraic
properties -- the spectral moments -- to structural features of the network,
namely the degree sequence, the number of cycles of length 3 and 5, and all
the correlation terms defined in (\ref{Correlation terms}). The key steps
behind the proof are: ($i$) Relate the spectral moments $m_{4}\left( L_{%
\mathcal{G}}\right) $ and $m_{5}\left( L_{\mathcal{G}}\right) $ with closed walks of length
four and five in the Laplacian graph $L\left( \mathcal{G}\right) $, and ($ii$) classify the set of closed
walks in $L\left( \mathcal{G}\right) $ into subsets according to
the subgraph covered by each walk.
\end{remark}

%\begin{figure*}[t]
%\centering
%\includegraphics[width=1.0\textwidth]{Fig5PNG.png}
%\caption{Possible types of closed walks of length $5$ in a simple graph. The
%classification is based on the structure of the subgraph underlying the
%closed walks. For each walk type, we also include an expression that
%corresponds to the number of closed walks of that particular type in terms
%of network structural features.}
%\label{fig_5}
%\end{figure*}

In the next section, we present a series of semidefinite programs (SDP's)
whose solutions provide optimal bounds on the Laplacian spectral radius and
spectral gap in terms of Laplacian spectral moments.

\section{\label{Optimal Bounds}Optimal Laplacian Bounds from Spectral Moments%
}

In this section, we introduce a novel approach to compute bounds on the
spectral gap and the spectral radius of the Laplacian matrix from a
truncated sequence of Laplacian spectral moments. More explicitly, the
problem solved in this section can be stated as follows:\medskip

\begin{problem}[Moment-based bounds]
\label{Lasserre problem}Given a truncated sequence of Laplacian spectral
moments $\left( m_{k}\left( L_{\mathcal{G}}\right) \right) _{k=1}^{K}$, find
bounds on the spectral gap and the spectral radius of the Laplacian matrix $%
L_{\mathcal{G}}$.
\end{problem}

\medskip

In this section, we propose a solution to the above problem based on a
recent result in \cite{Las11}. In \cite{Las11}, Lasserre developed an
approach to find bounds on the support of an unknown density function when
only a sequence of its moments is available. In order to adapt Problem 1 to
this framework, we need to introduce some definitions. Given a simple
connected graph $\mathcal{G}$ with Laplacian eigenvalues $\left\{ \lambda
_{i}\right\} _{i=1}^{n}$, we define the \emph{spectral density} of the
nontrivial eigenvalue spectrum as%
\begin{equation}
\rho _{\mathcal{G}}\left( \lambda \right) \triangleq \frac{1}{n-1}%
\sum_{i\geq 2}\delta \left( \lambda -\lambda _{i}\right) ,
\label{Spectral Measure}
\end{equation}%
where $\delta \left( \cdot \right) $ is the Dirac delta function. Notice how
we have excluded the trivial eigenvalue, $\lambda _{1}=0$, from the spectral
density; hence, the support \footnote{%
Recall that the support of a density function $\mu $ on $\mathbb{R}$,
denoted by $supp\left( \mu \right) $, is the smallest closed set $B$ such
that $\mu \left( \mathbb{R}\backslash B\right) =0$.} of $\rho _{\mathcal{G}%
}\left( \lambda \right) $ is equal to $supp(\rho _{\mathcal{G}})=\left\{
\lambda _{i}\right\} _{i=2}^{n}$. The moments of the spectral density in (%
\ref{Spectral Measure}), denoted by $\overline{m}_{k}\left( L_{\mathcal{G}%
}\right) $, can be written in terms of the spectral moments of $L_{\mathcal{G%
}}$, as follows%
\begin{align}
\overline{m}_{k}\left( L_{\mathcal{G}}\right) & \triangleq \int_{\mathbb{R}%
}\lambda ^{k}\frac{1}{n-1}\sum_{i=2}^{n}\delta \left( \lambda -\lambda
_{i}\right) ~d\lambda  \notag \\
& =\frac{1}{n-1}\sum_{i=2}^{n}\lambda _{i}^{k}=\frac{n}{n-1}m_{k}\left( L_{%
\mathcal{G}}\right) ,  \label{Nontrivial Moments}
\end{align}%
for all $k\geq 1$ (where we have used the fact that $\lambda _{1}=0$, in our
derivations).

In what follows, we propose a solution to Problem \ref{Lasserre problem}
using a technique proposed by Lasserre in \cite{Las11}. In that paper, the
following problem was addressed:

\medskip

\begin{problem}
\label{Lasserre Paper Problem}Consider a truncated sequence of moments $%
\left( M_{k}\right) _{1\leq r\leq K}$ corresponding to an unknown density
function $\mu \left( \lambda \right) $, i.e., $M_{k}\triangleq \int \lambda
^{k}d\mu \left( \lambda \right) $. Denote by $\left[ a,b\right] $ the
smallest interval containing the support of $\mu $. Compute an upper bound $%
\alpha \geq a$ and a lower bound $\beta \leq b$ when only the truncated
sequence of moments is available.
\end{problem}

\medskip

In the context of Problem 1, we have access to a truncated sequence of five
spectral moments, $(m_{k}(L_{\mathcal{G}}))_{1\leq k\leq 5}$, corresponding
to the unknown spectral density function $\rho _{\mathcal{G}}$ and given by
the expressions (\ref{Moments as Averages}), (\ref{Fourth and Fifth Moments}%
), and (\ref{Nontrivial Moments}). In this context, the smallest interval $%
[a,b]$ containing $supp(\rho _{\mathcal{G}})$ is equal to $[\lambda
_{2},\lambda _{n}]$. Therefore, a solution to Problem \ref{Lasserre Paper
Problem} would directly provide an upper bound on the spectral gap, $\alpha
\geq \lambda _{2}$, and a lower bound on the spectral radius, $\beta \leq
\lambda _{n}$. We now describe a numerical scheme proposed in \cite{Las11}
to solve Problem \ref{Lasserre Paper Problem}. This solution is based on a
series of semidefinite programs in one variable. In order to formulate this
series of SDP's, we need to introduce some definitions. For any $s\in 
\mathbb{N}$, let us consider a truncated sequence of moments $\mathbf{M}%
=\left( M_{k}\right) _{k=1}^{2s+1}$, associated with an unknown density
function $\mu $. We define the following Hankel matrices of moments:%
\begin{equation}
R_{2s}\left( \mathbf{M}\right) \triangleq \left[ 
\begin{array}{cccc}
1 & M_{1} & \cdots & M_{s} \\ 
M_{1} & M_{2} & \cdots & M_{s+1} \\ 
\vdots & \vdots & \ddots & \vdots \\ 
M_{s} & M_{s+1} & \cdots & M_{2s}%
\end{array}%
\right] ,\text{ }R_{2s+1}\left( \mathbf{M}\right) \triangleq \left[ 
\begin{array}{cccc}
M_{1} & M_{2} & \cdots & M_{s+1} \\ 
M_{2} & M_{3} & \cdots & M_{s+2} \\ 
\vdots & \vdots & \ddots & \vdots \\ 
M_{s+1} & M_{s+2} & \cdots & M_{2s+1}%
\end{array}%
\right] .  \label{Hankel Matrices}
\end{equation}%
We also define the \emph{localizing matrix}\footnote{%
A more general definition of localizing matrix can be found in \cite{LasBOOK}%
. For simplicity, we restrict our definition to the particular form used in
our problem.} $H_{s}\left( x,\mathbf{M}\right) $ as,%
\begin{equation}
H_{s}\left( x,\mathbf{M}\right) \triangleq R_{2s+1}\left( \mathbf{M}\right)
-x~R_{2s}\left( \mathbf{M}\right) .  \label{Localizing matrix}
\end{equation}

Using the above matrices, Lasserre proposed in \cite{Las11} the following
series of SDP's to find a solution for Problem \ref{Lasserre Paper Problem}:

\medskip

\emph{Solution to Problem 2}: Let $\mathbf{M}=\left( M_{k}\right)
_{k=1}^{2s+1}$ be a truncated sequence of moments associated with an unknown
density function $\mu $. Then%
\begin{align}
a& \leq \alpha _{s}\left( \mathbf{M}\right) \triangleq \max_{x}\left\{
x:H_{s}\left( x,\mathbf{M}\right) \succeq 0\right\} ,
\label{Bound min eigenval} \\
b& \geq \beta _{s}\left( \mathbf{M}\right) \triangleq \min_{x}\left\{
x:-H_{s}\left( x,\mathbf{M}\right) \succeq 0\right\} ,
\label{Bound max eigenval}
\end{align}%
where $\left[ a,b\right] $ is the smallest interval containing the support
of $\mu $.

\medskip

Therefore, we can directly apply the above result to solve Problem \ref%
{Lasserre problem} by considering the sequence of moments $\overline{\mathbf{%
m}}\triangleq \left( \overline{m}_{r}\left( L_{\mathcal{G}}\right) \right)
_{r=1}^{2s+1}=(\frac{n}{n-1}m_{r}\left( L_{\mathcal{G}}\right)
)_{r=1}^{2s+1} $ in the statement of the solution to Problem \ref{Lasserre
Paper Problem}. Since this sequence of moments corresponds to the spectral
density $\rho _{\mathcal{G}}$, with support $\{\lambda _{i}\}_{i=2}^{n}$,
the solutions in (\ref{Bound min eigenval}) and (\ref{Bound max eigenval})
directly provide the following bounds on the spectral radius and spectral
gap:

\emph{Solution to Problem 1}: Let $\overline{\mathbf{m}}\triangleq (\frac{n}{%
n-1}m_{r}\left( L_{\mathcal{G}}\right) )_{r=1}^{2s+1}$ be a truncated
sequence of (scaled) Laplacian spectral moments associated with a graph $%
\mathcal{G}$. Then the Laplacian spectral gap and spectral radius of $%
\mathcal{G}$ satisfy the following bounds:%
\begin{align}
\lambda _{2}& \leq \alpha _{s}\left( \overline{\mathbf{m}}\right) \triangleq
\max_{x}\left\{ x:H_{s}\left( x,\overline{\mathbf{m}}\right) \succeq
0\right\} ,  \label{Bound Spectral Gap} \\
\lambda _{n}& \geq \beta _{s}\left( \overline{\mathbf{m}}\right) \triangleq
\min_{x}\left\{ x:-H_{s}\left( x,\overline{\mathbf{m}}\right) \succeq
0\right\} .  \label{Bound Spectral Radius}
\end{align}

In Section \ref{Laplacian Spectral Analysis}, we derived expressions for the
first five Laplacian spectral moments, $\left( m_{r}\left( L_{\mathcal{G}%
}\right) \right) _{r=1}^{5}$, in terms of structural features of the
network, namely, the degree sequence, the number of triangles, quadrangles
and pentagons, and the correlation terms in (\ref{Correlation terms}).
Therefore, we can apply the \emph{Solution to Problem 1} to find bounds on $%
\lambda _{2}$ and $\lambda _{n}$.

In this section, we have presented an optimization-based approach to compute
optimal bounds on the Laplacian spectral gap and spectral radius from a
truncated sequence of Laplacian spectral moments. The truncated sequence of
spectral moments $\left( m_{k}\left( L_{\mathcal{G}}\right) \right)
_{k=1}^{5}$ can be written in terms of local structural measurements using (%
\ref{Moments as Averages}) and (\ref{Fourth and Fifth Moments}). Hence, the
above methodology allows to compute bounds on the spectral radius and
spectral gap of the Laplacian matrix given a collection of local structural
features of the network. In the following section, we illustrate the usage
of this approach with numerical examples.

\section{\label{Simulations}Structural Analysis and Simulations}

In this section, we apply the moment-based approach herein proposed
to study the relationship between structural and spectral properties of an
unweighted, undirected graph representing the structure of the high-voltage
transmission network of Spain (the adjacency of this network is available,
in MATLAB\ format, in \cite{MAT1}). The number of nodes (buses) and edges
(transmission lines) in this network are $n=98$ and $e=175$, respectively. From this dataset, we compute the set of structural
properties involved in (\ref{Moments as Averages}) and (\ref{Fourth and
Fifth Moments}), namely, the power-sums of the degrees $\left( S_{r}\right)
_{r=1}^{5}=\left( 350,1692,9836,64056,44942\right) $, the number of
cycles $\Delta =79$, $Q=134$, $P=232$, and the
correlation terms $C_{dd}=42.58$, $C_{d^{2}d}=249.41$, $C_{dt}=13.98$%
, $C_{d^{2}t}=88.69$, $C_{dq}=33.11$, and $D_{dd}=80.77$. Using this collection of structural
measurements, we use (\ref{Moments as Averages}) and (\ref{Fourth and Fifth
Moments}) to compute the first five Laplacian spectral moments of the
Spanish transmission network: $\left( m_{k}\left( L_{\mathcal{G}}\right)
\right) _{k=1}^{5}=\left( 3.571,20.83,147.33,1155.5,9686.6\right) $.
Using this sequence of spectral moments and the methodology described in
Section \ref{Optimal Bounds}, we compute bounds on the spectral gap and
spectral radius, $\alpha _{2}$ and $\beta _{2}$, solving
the SDP's in (\ref{Bound Spectral Gap}) and (\ref{Bound Spectral Radius}).
The numerical values for these bounds, as well as the exact values for the
spectral gap and spectral radius are: $\beta _{2}=9.18\leq \lambda
_{n}=10.66$ and $\lambda _{2}=0.077\leq \alpha
_{2}=0.86$.

Our numerical analysis reveals that the Laplacian spectral radius
and spectral moments of the electrical transmission network are strongly
constrained by local structural features of the network. On the other hand,
the spectral gap cannot be efficiently bounded using local structural
features only, since the spectral gap strongly depends on the global
connectivity of the network. This limitation is inherent to all spectral
bounds based on local structural properties (see \cite{Abreu07} for a wide
collection of spectral bounds). In the following example, we illustrate this
limitation with a simple example.

\begin{example}
Consider a ring graph with $n=12$ nodes, which we denote
by $R_{12}$. The eigenvalues of the Laplacian matrix of a ring
graph of length $l$, $R_{l}$, are equal to $\lambda _{i}=2-2\cos (2\pi i/l)$%
, for $i=0,...,l-1$ \cite{Big93}. Therefore, the Laplacian spectral gap and spectral radius of $R_{12}$ %
are $\lambda _{2}=2-2\cos \pi /6\approx 0.2679$ and $\lambda _{n}=4$%
, respectively. We can also compute the moment-based
bounds $\alpha _{2}$ and $\beta _{2}$ using local
structural measurements, as follows. The degrees of all the nodes in $%
R_{12} $ are $d_{i}=2$; thus, the power sums of the
degrees are equal to $S_{k}=2^{k}12$. The number of triangles,
quadrangles and pentagons are $\Delta =Q=P=0$. The correlation
terms are $C_{dd}=4$, $C_{d^{2}d}=8$, and the rest of
correlation terms in (\ref{Correlation terms}) are equal to zero. Based on
these structural measurements, we have from (\ref{Moments as Averages}) and (\ref{Fourth and Fifth Moments}) that the first five Laplacian spectral
moments are $\left( m_{r}\left( L_{\mathcal{G}}\right) \right)
_{r=1}^{5}=\left( 2,6,20,70,252\right) $, and the resulting
moment-based bounds from (\ref{Bound Spectral Gap}) and (\ref{Bound Spectral
Radius})\ are $\beta _{2}=3.732\leq 4$ and $%
\alpha _{2}=0.2679\approx \lambda _{2}$. Therefore, both the bounds on the
spectral radius and the spectral gap are very tight for $R_{12}$.
In particular, $\alpha _{2}$ is remarkably close
to $\lambda _{2}$.

On the other hand, we can construct graphs with the same local
structural properties (and, therefore, the same first five spectral moments,
and bounds $\alpha _{2}$ and $\beta _{2}$), but very
different spectral gap, as follows. Consider a graph of $12$ nodes
consisting in two disconnected rings of length 6. It is easy to verify that
this (disconnected) graph presents the same local structural features as a
connected ring of length 12, namely, the degrees of all the nodes
are $d_{i}=0$, the number of cycles $\Delta=Q=P=0$, and
the correlation terms are the same as the ones computed above. In contrast to $R_{12}$, 
the spectral gap of this disconnected graph is equal
to zero, $\lambda _{2}=0$, which is very different than the
moment-based bound $\alpha _{2}$.
\end{example}

In general, the Laplacian spectral gap is a global property that
quantifies how `well-connected' a network is \cite{KM06}. Since the
structural measurements used in our bounds (degree sequence, correlation
terms, etc.) have a local nature, they do not contain enough information to
determine how well connected the network is globally. In other words, it is
often possible to find two different graphs with identical local structural
features but radically different global structure, as we have illustrated in
the above example.

\section{Conclusions}

This paper studies the relationship between \emph{local} structural
features of large complex networks and \emph{global }spectral properties of
their Laplacian matrices. In Section \ref{Laplacian Spectral Analysis}, we
have proposed a graph-theoreical approach to compute the first five
Laplacian spectral moments of a network from a collection of local
structural measurements. In Section \ref{Optimal Bounds}, we have proposed
an optimization-based approach, based on a recent result by Lasserre \cite{Las11}, 
to compute bounds on the Laplacian spectral radius and spectral gap
of a network from a truncated sequence of spectral moments. Our bounds take
into account the effect of important structural properties that are usually
neglected in most of the bounds found in the literature, such as the
distribution of cycles and other structural correlations. Our analysis shows
that local structural features of the network strongly constrain the
Laplacian spectral moments and spectral radius. On the other hand, local
structural features are not enough to characterize the Laplacian spectral
gap, since this quantity strongly depends on how `well-connected' the
network is globally.

\appendix

\section{Proof of Lemma \protect\ref{Lemma 5th metrics}}

\emph{Theorem} \ref{High Order Laplacian Moments} %
Let $\mathcal{G}$ be a simple graph with Laplacian matrix $L_{\mathcal{G}}$.
Then, the fourth and fifth Laplacian moments can be written as%
\begin{eqnarray*}
m_{4}\left( L_{\mathcal{G}}\right) &=&\frac{1}{n}\left(
-S_{1}+2S_{2}+4S_{3}+S_{4}+8Q\right) \\
&&+4C_{dd}-8C_{dt}, \\
m_{5}\left( L_{\mathcal{G}}\right) &=&\frac{1}{n}\left(
-5S_{2}+5S_{3}+5S_{4}+S_{5}+30\Delta -10P\right) \\
&&+10\left( C_{dd}+C_{d^{2}d}-C_{dt}-C_{d^{2}t}+C_{dq}-D_{dd}\right)
\end{eqnarray*}%
where $S_{r}=\sum_{v_{i}\in \mathcal{V}}d_{i}^{r}$, and the correlation
terms $C_{dd},$ $C_{dt},$ $C_{dq},$ $C_{d^{2}d},$ $C_{d^{2}t},$ and $D_{dd}$
are defined in (\ref{Correlation terms}).

\begin{proof}
As in Theorem \ref{Low Order Laplacian Moments}, we use Lemma \ref{Moments
in Weighted Graphs} to compute the Laplacian spectral moments in terms of
weighted sums of closed walks in the weighted Laplacian graph $\mathcal{L}%
_{G}$. In order to compute the fourth Laplacian spectral moment, we classify
the types of possible closed walks of length 4 into subsets according to the
structure of the underlying graph covered by the walk. Specifically, two
walks $p_{1}$ and $p_{2}$ belong to the same type if the subgraphs covered
by the walks, denoted by $C\left( p_{1}\right) $ and $C\left( p_{2}\right) $
according to Definition \ref{Covered Subgraph}, are isomorphic. We enumerate
the possible types in Fig. \ref{fig_8} and we denote the corresponding sets
of walks as $P_{4a}^{\left( i\right) }$, $P_{4b}^{\left( i\right) }$, $%
P_{4c}^{\left( i\right) }$, $P_{4d}^{\left( i\right) }$, and $P_{4e}^{\left(
i\right) }$. These sets $P_{4a}^{\left( i\right) }$,...,$P_{4e}^{\left(
i\right) }$ partition the set of closed walks $P_{4,n}^{\left( i\right) }$.
Hence, we have $m_{4}\left( L_{\mathcal{G}}\right) =\frac{1}{n}%
\sum_{v_{i}\in \mathcal{V}}\sum_{x\in \left\{ a,b,c,d,e\right\} }\sum_{p\in
P_{4x}^{\left( i\right) }}\omega \left( p\right) .$

\begin{figure}[t]
\centering\includegraphics[width=1.0\textwidth]{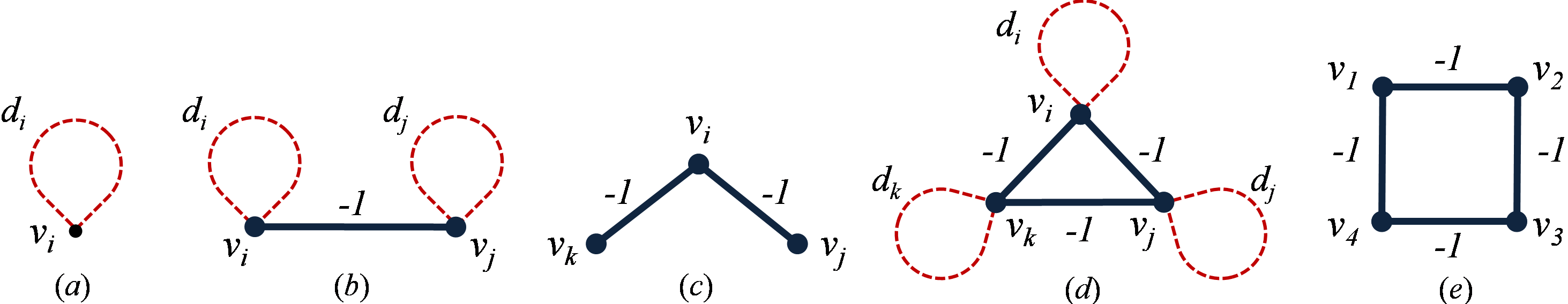}
\caption{Collection of possible graphs covered by closed walks of length 4.}
\label{fig_8}
\end{figure}

We now analyze each one of the terms in the above summations. For
convenience, we define $T_{4x}\triangleq \frac{1}{n}\sum_{v_{i}\in \mathcal{V%
}}\sum_{p\in P_{4x}^{\left( i\right) }}\omega \left( p\right) ,$ and analyze
the term $T_{4x}$ for $x\in \left\{ a,b,c,d\right\} $:

(\emph{a}) For $x=a$, we have that the weights $\omega \left( p\right) $ of
the walks in $P_{4a}^{\left( i\right) }$ are all the same, and equal to $%
d_{i}^{4}$. Hence, $T_{4a}=\frac{1}{n}\sum_{i}d_{i}^{4}=S_{4}/n.$

(\emph{b}) For $x=b$, the weights of the walks in $P_{4b}^{\left( i\right) }$
are equal to $2+4\left( d_{i}^{2}+d_{j}^{2}+d_{i}d_{j}\right) $. Hence, $%
T_{4b}=\frac{1}{n}\sum_{v_{i}\sim v_{j}}2+4\left(
d_{i}^{2}+d_{j}^{2}+d_{i}d_{j}\right) =\frac{1}{n}\left( S_{1}+4S_{3}\right)
+4C_{dd}.$

(\emph{c}) For $x=c$, the weights of the walks in $P_{4c}^{\left( i\right) }$
(i.e., walks that cover the two-chain graph) are equal to $4$. Hence, $%
T_{4c}=\frac{1}{n}\sum_{v_{j}\sim v_{i}\sim v_{k}}4\overset{(i)}{=}\frac{1}{n%
}\sum_{i=1}^{n}\binom{d_{i}}{2}4=\frac{2}{n}\left( S_{2}-S_{1}\right) ,$
where in equality (\emph{i}) we have used the fact that the number of
two-chain graphs whose center node is $v_{i}$ is equal to $\binom{d_{i}}{2}$.

\begin{figure*}[t]
\centering\includegraphics[width=1.0\textwidth]{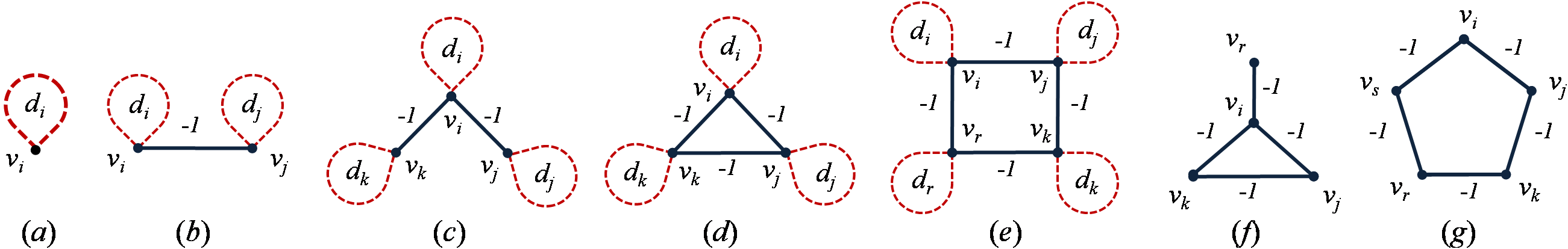}
\caption{Collection of possible graphs covered by closed walks of length 5.}
\label{fig_9}
\end{figure*}

(\emph{d}) For $x=d$, the weights of the walks in $P_{4d}^{\left( i\right) }$
are equal to $-8\left( d_{i}+d_{j}+d_{k}\right) $. Hence, $T_{4d}=\frac{1}{n}%
\sum_{v_{i}\sim v_{j}\sim v_{k}\sim v_{i}}-8\left( d_{i}+d_{j}+d_{k}\right)
=-\frac{8}{n}\sum_{i=1}^{n}\sum_{j=1}^{n}\sum_{k=1}^{n}3t_{ijk}d_{i},$where $%
t_{ijk}$ is an indicator function that takes value 1 if $v_{i}\sim v_{j}\sim
v_{k}\sim v_{i}$. Since $\sum_{j=1}^{n}\sum_{k=1}^{n}3t_{ijk}=t_{i}$ (the
number of triangles touching node $v_{i}$), we have that $T_{4d}=-\frac{8}{n}%
\sum_{i=1}^{n}t_{i}d_{i}=-8C_{dt}.$

(\emph{e}) For $x=e$, the weights of the walks in $P_{4e}^{\left( i\right) }$
are equal to $8$. Hence, $T_{4e}=\frac{1}{n}\sum_{\substack{ v_{i}\sim
v_{j}\sim v_{k}\sim v_{r}\sim v_{i}  \\ \text{s.t. }1\leq i<j<k<r\leq n}}%
8=8Q/n.$

Finally, since $m_{4}\left( L_{\mathcal{G}}\right)
=T_{4a}+T_{4b}+T_{4c}+T_{4d}$, we obtain the expression for the fourth
Laplacian spectral moment in the statement of the theorem after simple
algebraic simplifications.

In order to derive a similar expression for the fifth-order Laplacian
spectral moments, we follow an identical approach. Below, we provide the
main steps in the derivations. As before, we partition the set of closed
walks $P_{5,n}^{\left( i\right) }$ according to the subgraph covered by the
walk. We show the structure of the possible subgraphs in Fig. \ref{fig_9}.

We now analyze each one of the terms $T_{5x}\triangleq \frac{1}{n}%
\sum_{v_{i}\in \mathcal{V}}\sum_{p\in P_{5x}^{\left( i\right) }}\omega
\left( p\right) $ for $x\in \left\{ a,b,...,g\right\} $:

(\emph{a}) For $x=a$, we have $T_{5a}=\frac{1}{n}%
\sum_{i=1}^{n}d_{i}^{5}=S_{5}/n.$

(\emph{b}) For $x=b$, we can determine all possible closed walks of length 5
using the edge graph in Fig. \ref{fig_9}(b) and derive that $T_{5b}=\frac{1}{%
n}\sum_{v_{i}\sim v_{j}}5\left(
d_{i}+d_{j}+d_{i}^{3}+d_{j}^{3}+d_{i}^{2}d_{j}+d_{i}d_{j}^{2}\right) =\frac{5%
}{n}\left( S_{2}+S_{4}\right) +10C_{d^{2}d}.$

(\emph{c}) For $x=c$, the weights of walks covering the two-chain graph are $%
d_{i}$, $d_{j}$, $d_{k}$. Counting the multiplicities of each type of walk,
we have that $T_{5c}=\frac{1}{n}\sum_{v_{j}\sim v_{i}\sim
v_{k}}10d_{i}+5d_{j}+5d_{k}=\frac{10}{n}\sum_{i=1}^{n}\binom{d_{i}}{2}d_{i}+%
\frac{5}{n}\sum_{i=1}^{n}\sum_{j=1}^{n}a_{ij}\left( d_{i}-1\right) d_{j},$%
where we have used that $\sum_{v_{i}\sim v_{j}\sim v_{k}}d_{i}=\sum_{i=1}^{n}%
\binom{d_{i}}{2}d_{i}$ and $\sum_{v_{j}\sim v_{i}\sim v_{k}}d_{j}=\frac{1}{2}%
\sum_{i=1}^{n}\sum_{j=1}^{n}a_{ij}\left( d_{i}-1\right) d_{j}$. Thus, $%
T_{5c}=\frac{5}{n}\sum_{i=1}^{n}d_{i}^{3}-\frac{5}{n}\sum_{i=1}^{n}d_{i}^{2}+%
\frac{5}{n}\sum_{1\leq i,j\leq n}a_{ij}d_{i}d_{j}-\frac{5}{n}%
\sum_{j=1}^{n}d_{j}^{2}=\frac{5}{n}\left( S_{3}-2S_{2}\right) +10C_{dd}$

(\emph{d}) For $x=d$, we can determine all possible closed walks of length 5
using the edge graph in Fig. \ref{fig_9}(d) and derive that where $%
b_{ij}\triangleq $ $\sum_{k=1}^{n}a_{ik}a_{jk}=\left\vert \mathcal{N}%
_{i}\cap \mathcal{N}_{j}\right\vert $, the number of common neighbors shared
by $v_{i}$ and $v_{j}$. Hence, $T_{5d}=-\frac{30\Delta }{n}-\frac{10}{n}%
\sum_{i=1}^{n}t_{i}d_{i}^{2}-\frac{10}{n}\sum_{i\sim j}a_{ij}b_{ij}\left(
d_{i}d_{j}\right) =-30\Delta /n-10C_{d^{2}t}-10D_{dd}$

(\emph{e}) For $x=e$, the weights of walks covering the quadrangle graph are 
$d_{i}$, $d_{j}$, $d_{k}$, and $d_{r}$. Counting the multiplicities of each
type of walk we have that $T_{5e}=\frac{1}{n}\sum_{v_{i}\sim v_{j}\sim
v_{k}\sim v_{r}\sim v_{i}}10\left( d_{i}+d_{j}+d_{k}+d_{r}\right) =\frac{10}{%
n}\sum_{i=1}^{n}\sum_{j=1}^{n}\sum_{k=1}^{n}\sum_{r=1}^{n}4q_{ijkr}d_{i},$%
where $q_{ijkr}$ is an indicator function that takes value 1 if $v_{i}\sim
v_{j}\sim v_{k}\sim v_{r}\sim v_{i}$. Since $\sum_{1\leq j,k,r\leq
n}4q_{ijkr}=q_{i}$ (the number of quadrangles touching node $v_{i}$), we
have that $T_{5e}=\frac{10}{n}\sum_{i=1}^{n}q_{i}d_{i}=10C_{dq}.$

(\emph{f}) For $x=f$, we have 10 possible walks covering the subgraph in
Fig. \ref{fig_9}(f). Since each walks has a weight equal to $-1$, we have
that $T_{5f}=\frac{1}{n}\sum_{v_{i}\sim v_{j}\sim v_{k}\sim v_{i}\sim
v_{r}}-10=-\frac{10}{n}\sum_{i=1}^{n}\left( d_{i}-2\right) t_{i},$ where in
the last equality we take into account that the number of subgraphs of the
type depicted in Fig. \ref{fig_9}(f) and centered at node $v_{i}$ is equal
to the number of triangles touching node $v_{i}$, $t_{i}$, multiplied by $%
\left( d_{i}-2\right) $ (where we have subtracted $-2$ to the degree to
discount the two edges touching $v_{i}$ that are part of each triangle
counted in $t_{i}$). Hence, we have that $T_{5f}=-\frac{10}{n}%
\sum_{i=1}^{n}d_{i}t_{i}+\frac{10}{n}\sum_{i=1}^{n}2t_{i}=-10C_{dt}+60\Delta
/n.$

(\emph{g}) For $x=f$, we have 10 possible walks on the pentagon and the
associated weight of each walk is $-1$. Hence, $T_{5g}=\frac{1}{n}%
\sum_{v_{i}\sim v_{j}\sim v_{k}\sim v_{r}\sim v_{s}\sim v_{i}}-10=-10P/n,$%
where $P$ is the total number of pentagons in $\mathcal{G}$.

Finally, since $m_{5}\left( L_{\mathcal{G}}\right) =T_{5a}+T_{5b}+...+T_{5g}$%
, we obtain the expression for the fifth Laplacian spectral moment in the
statement of the theorem after simple algebraic simplifications.\bigskip
\end{proof}

\end{document}